\documentclass[letterpaper, 10 pt, conference]{ieeeconf}

\IEEEoverridecommandlockouts
	\usepackage{adjustbox}
	\usepackage{algorithm}
	\usepackage[noend]{algpseudocode}
	\usepackage{amsfonts}
	\usepackage{amsmath}
	\usepackage{amssymb}
	\usepackage{amsthm}
	\usepackage{bm}
	\usepackage{comment}
	\let\labelindent\relax
	\usepackage{enumitem}
	\usepackage{graphicx}
	\usepackage{mathtools}
	\usepackage{multirow}
	\usepackage{nicefrac}
	\usepackage{pgfplots}
		\pgfplotsset{compat=1.18}
	\usepackage{pgfplotstable}
	\usepackage{tikz}

	\makeatletter
		\let\NAT@parse\undefined
	\makeatother
	\usepackage{hyperref}
		\hypersetup{
			colorlinks=true,
			linkcolor=blue,
			filecolor=magenta,
			citecolor =magenta,
			urlcolor=magenta,
			pdftitle={PANOC-lite},
			breaklinks,
			hypertexnames = true,
			plainpages    = false,
		}
	\usepackage[capitalize,nameinlink]{cleveref}
	\AtBeginDocument{%
		\crefformat{section}{#2\S#1#3}%
		\crefformat{subsection}{#2\S#1#3}%
		\Crefformat{section}{#2\S#1#3}%
		\Crefformat{subsection}{#2\S#1#3}%
	}%

	\usepackage{xstring}
	\let\oldcref\cref
	\let\oldCref\Cref
	\renewcommand{\cref}[1]{%
		\IfStrEq{alg:zerofprpp}{#1}{\panocpp}{\oldcref{#1}}%
	}%
	\renewcommand{\Cref}[1]{%
		\IfStrEq{alg:zerofprpp}{#1}{\Panocpp}{\oldCref{#1}}%
	}%
	\DeclarePairedDelimiter\norm{\lVert}{\rVert}
	\DeclarePairedDelimiter\abs{\lvert}{\rvert}

	\DeclareMathOperator*{\argmin}{arg\,min}

	\DeclareMathOperator*{\minimize}{minimize}
	
	\DeclareMathOperator{\stt}{subject\ to}

	\DeclareMathOperator{\dist}{dist}     
	\DeclareMathOperator{\fix}{fix}       
	\newcommand{\id}{{\rm id}}            
	\DeclareMathOperator{\lip}{lip}       
	\DeclareMathOperator{\prox}{prox}     
	\DeclareMathOperator{\proj}{\Pi}      
	\DeclareMathOperator{\relint}{relint} 
	\DeclareMathOperator{\zer}{zer}       

	\newcommand{\Rnat}[1]{\operatorname{R}_{#1}^{\rm nat}}
	\newcommand{\Rnor}[1]{\operatorname{R}_{#1}^{\rm nor}}

	\newcommand\nnz{{n_\text{z}}}
	\newcommand\nnu{{n_\text{u}}}
	\newcommand\nnc{{n_\text{c}}}

	\newcommand{\exR}{\overline{\mathbb{R}}}

	\newcommand{\cJ}{\mathcal{J}}
	\newcommand{\cK}{\mathcal{K}}

	\newcommand{\R}{\mathbb{R}}
	\newcommand{\N}{\mathbb{N}}
	\newcommand{\C}{\mathcal{C}}
	\newcommand{\fbe}{\varphi_\gamma}
	\newcommand{\eqdef}{\coloneqq}

	\newcommand{\jac}[1]{\mathrm{J}\,{#1}}

	\newcommand{\seq}[1]{(#1)_{k\in\N}}

	\newcommand{\alpaqa}{{\sc alpaqa}}
	
	\newcommand{\panoc}{{\sc panoc}}
	\newcommand{\Panoc}{{\sc Panoc}}
	\NewDocumentCommand{\panocpp}{s}{%
		\IfBooleanTF{#1}{{\sc panoc}{\rm-lite}}{%
			\hyperref[alg:zerofprpp]{{\sc panoc}{\rm-lite}}%
		}%
	}
	\NewDocumentCommand{\Panocpp}{s}{%
		\IfBooleanTF{#1}{{\sc Panoc}{\rm-lite}}{%
			\hyperref[alg:zerofprpp]{{\sc Panoc}{\rm-lite}}%
		}%
	}
	\NewDocumentCommand{\zerofprpp}{s}{%
		\IfBooleanTF{#1}{{\sc panoc}{\rm-lite}}{%
			\hyperref[alg:zerofprpp]{{\sc panoc}{\rm-lite}}%
		}%
	}
	\NewDocumentCommand{\Zerofprpp}{s}{%
		\IfBooleanTF{#1}{{\sc Panoc}{\rm-lite}}{%
			\hyperref[alg:zerofprpp]{{\sc Panoc}{\rm-lite}}%
		}%
	}
	\newcommand{\zerofpr}{{\sc ZeroFPR}}

	\newcommand{\Cpp}{C\kern-0.04em+\kern-0.01em+}

	\newcommand{\algorithmicinitialize}{\textsc{Initialize:}}
	\algnewcommand{\Initialize}{\item[\algorithmicinitialize]}%

	\usepackage{eqparbox}
	\newcommand{\listlabel}[2][l]{\eqmakebox[listlabel@\EnumitemId][#1]{#2}}
	\SetEnumitemKey{widestL}{
		labelwidth=\eqboxwidth{listlabel@\EnumitemId},
		leftmargin=!,
		format=\listlabel,
	}
	\SetEnumitemKey{thm}{%
		label=\textup{(\roman*)},
		ref=\thetheorem(\roman*),
		align=left,
		widestL,
	}

	\newtheorem{theorem}{Theorem}
		\newlist{theoremenum}{enumerate}{1} 
		\setlist[theoremenum]{thm}
		\crefalias{theoremenumi}{theorem}

	\newtheorem{assumption}{Assumption}
		\renewcommand{\theassumption}{\Roman{assumption}}
		\newlist{assumenum}{enumerate}{1} 
		\setlist[assumenum]{label=\textup{(\roman*)}, ref=\theassumption(\roman*), widestL}
		\crefalias{assumenumi}{assumption}

	\newtheorem{corollary}[theorem]{Corollary}
		\newlist{corollenum}{enumerate}{1} 
		\setlist[corollenum]{thm}
		\crefalias{corollenumi}{corollary}

	\newtheorem{lemma}[theorem]{Lemma}
		\newlist{lemenum}{enumerate}{1} 
		\setlist[lemenum]{thm}
		\crefalias{lemenumi}{lemma}

	\newtheorem{proposition}[theorem]{Proposition}
		\newlist{propenum}{enumerate}{1} 
		\setlist[propenum]{thm}
		\crefalias{propenumi}{proposition}

	\theoremstyle{definition}
	\newtheorem{remark}[theorem]{Remark}

	\usetikzlibrary{arrows.meta}
	\usetikzlibrary{backgrounds}
	\usepgfplotslibrary{patchplots}
	\usepgfplotslibrary{fillbetween}
	\usepgfplotslibrary{groupplots}

\title{PANOC-lite: A simpler and more efficient algorithm for composite minimization}

\author{%
	Alexander Bodard$^*$\and
	Pieter Pas$^*$\and
	Andreas Themelis$^\dagger$\and
	Panagiotis Patrinos$^*$\thanks{%
		This work was supported by: Research Foundation Flanders (FWO)
		PhD grant No. 11M9523N and research projects G081222N, G033822N,
		G0A0920N; Research Council KU Leuven C1 project No. C14/18/068; JSPS KAKENHI grant number JP24K20737.\newline
		$^*$ Department of Electrical Engineering (ESAT-STADIUS), KU Leuven, Kasteelpark Arenberg 10, 3001 Leuven, Belgium.\newline
		$^\dagger$ Faculty of Information Science and Electrical Engineering (ISEE),
		Kyushu University, 744 Motooka, Nishi-ku, 819-0395 Fukuoka, Japan.\newline
		{\it E-mails:}\hfill
		{\footnotesize\sf
			\{alexander.bodard,pieter.pas,panos.patrinos\}@kuleuven.be,\linebreak
			andreas.themelis@ees.kyushu-u.ac.jp%
		}%
	}%
}

\begin{document}
	\maketitle
	\thispagestyle{empty}
	\pagestyle{empty}


	\begin{abstract}
		This work introduces a simple and efficient linesearch method for composite minimization that accelerates proximal-gradient iterations with fast Newton-type directions.
Our algorithm is based on simple operations and only requires the standard proximal-gradient oracle, similar to PANOC and ZeroFPR, provided that the nonsmooth term is convex.
Noteworthy improvements include a cheaper backtracking procedure, in the sense that no additional gradients need to be evaluated, and an enlarged range of permitted stepsizes.
Global subsequential convergence and local superlinear convergence are established under conventional assumptions by considering a novel merit function which is less expensive to evaluate than alternatives like the forward-backward envelope.
Finally, the proposed approach is validated on model predictive control problems with collision avoidance constraints, as well as on the LIBSVM and CUTEst benchmarks.

	\end{abstract}

	\section{Introduction}
		\subsection{Problem statement and motivation}
	We consider nonconvex composite minimization problems
	\begin{equation*} \tag{P} \label{eq:problem-statement}
		\minimize_{x \in \R^n} \quad\varphi(x) \eqdef f(x) + g(x)
	\end{equation*}
	where \(f : \R^n \to \R\) is smooth, and \(g : \R^n \to \exR\) is convex and `proximable'.
	In particular, we assume the following:

	\begin{assumption}\label{assumption:basic}%
		The following hold in problem \eqref{eq:problem-statement}:
		\begin{assumenum}
			\item \label{assumption:basic-f}$f : \R^n \to \R$ is \(L_f\)-Lipschitz smooth;
			\item \label{assumption:basic-g}$g : \R^n \to \exR$ is proper, lsc and convex;
			\item \label{assumption:basic-minimizer}$\argmin \varphi \neq \emptyset$, i.e., a solution exists.
		\end{assumenum}
	\end{assumption}

	Composite objectives arise in various fields, like optimal control and machine learning.
	Prominent examples include constrained optimization, where \(g\) is the indicator of a closed convex set, and regularized problems, where \(g\) acts as a regularization term.
	More generally, nonlinear programs (NLPs) can also give rise to problems of the form \eqref{eq:problem-statement}:
	when solved using an augmented Lagrangian method (ALM), this requires solving a sequence of such composite subproblems.
	In the context of nonlinear model predictive control \cite{rawlings2017model}, this approach has been successfully applied, for example to problems with general state constraints \cite{pas2022alpaqa,sopasakis2020open}.

	A fundamental scheme in addressing \eqref{eq:problem-statement} is the \emph{proximal-gradient} (PG) method, with iterations of the form
	\begin{equation} \label{eq:pg}
		x^{k+1} = T_\gamma(x^k) \eqdef \prox_{\gamma g}(x^k - \gamma \nabla f(x^k))
	\end{equation}
	for some stepsize \(\gamma > 0\).
	This simple update rule serves as the basis for numerous accelerated variants, which either rely on Nesterov-type acceleration \cite{beck2009fast} or on fast Newton-type directions \cite{stella2017simple,themelis2018forward,ouyang2025trust,zeng2026linesearch}.
	However, ensuring global convergence of such methods often constitutes a significant challenge.
	A common approach is to design a merit function with suitable properties, and combine it with classical linesearch \cite{stella2017simple,themelis2018forward,zeng2026linesearch} or trust-region \cite{bodard2025second,ouyang2025trust} strategies.
	Initially proposed in \cite{patrinos2013proximal}, the forward-backward envelope (FBE) has been thoroughly analyzed and used as a merit function \cite{stella2017simple,themelis2018forward,bodard2025second}.
	Other surrogate functions give rise to alternative algorithms \cite{ouyang2025trust,zeng2026linesearch}.

	Somewhat remarkably, evaluating these existing merit functions not only requires function evaluations of $f$ and $g$, but also of $\nabla f$ and $\prox_{\gamma g}$.
	Therefore, every additional backtracking step of a linesearch procedure requires an additional evaluation of this first-order oracle, which is qualitatively different from the classical smooth setting.
	This raises the question whether we can design a merit function whose evaluation does not involve calls to $\nabla f$, which typically constitutes the most expensive operation.

\subsection{Contributions}
	To address this shortcoming,
	we propose a \emph{novel merit function} for composite objectives $\varphi \equiv f + g$ which, unlike existing alternatives, can be evaluated \emph{without gradient computations} $\nabla f$, and
	present a novel linesearch method (\cref{alg:zerofprpp}) that has a cheaper backtracking procedure and allows for larger stepsizes than state-of-the-art alternatives.
	This algorithm adopts the philosophy of \panoc{} \cite{stella2017simple}:%
	\footnote{Strictly speaking, the algorithm is closer to its predecessor \zerofpr{} \cite{themelis2018forward} in its way to compute update directions (see the overview in \cref{sec:background}). We nonetheless stick the more user-friendly name \panocpp.}
	despite only involving plain (nonscaled) proximal gradient operations,
	it attains global subsequential convergence and local superlinear convergence under conventional assumptions.
	Several numerical experiments on model predictive control problems with collision avoidance constraints, sparse logistic regression, and on \textsc{CUTEst} benchmarks demonstrate the performance of the proposed method.

\subsection{Notation}
	We denote by $\N$, $\R$ and $\exR = \R \cup \{ +\infty \}$ the set of natural, real and extended real numbers, respectively.
	The Euclidean inner product and norm are denoted by $\langle \cdot, \cdot \rangle$ and $\|\cdot\|$.
	By $\proj_C$ we denote the Euclidean projection on a set $C$.
	The set $\fix{T} \eqdef \{ x \in \R^n \mid x \in T(x) \}$ contains the fixed points of a set-valued operator $T : \R^n \rightrightarrows \R^n$ and $\zer T \eqdef \{x \mid 0 \in T(x)\}$ denotes its zeros.
	The proximal operator and Moreau envelope of a convex function $h : \R^n \to \exR$ with stepsize \(\gamma>0\) are $\prox_{\gamma h} (x) \eqdef \argmin_u \{ h(u) + \frac{1}{2\gamma} \norm{ u - x }^2 \}$ and $h^\gamma (x) \eqdef \inf_u \{ h(u) + \frac{1}{2\gamma} \norm{ u - x }^2 \}$, respectively.
	By \(\partial h\) we indicate the convex subdifferential of \(h\).
	We denote the set of $k$ times continuously differentiable functions by $\mathcal{C}^k$ and say that $f \in \mathcal{C}^1$ is $L_f$-smooth when $\nabla f$ is $L_f$-Lipschitz continuous.
    The pointwise Lipschitz modulus of \(h\) at \(x\in\R^n\) is \(\lip h(x)\coloneqq\limsup_{\substack{y,z\to x\\y\neq z}}\frac{\abs{h(y)-h(z)}}{\norm{x-z}}\).
	Otherwise, we adopt the notation from \cite{rockafellar1998variational}.

	\section{A novel merit function} \label{sec:merit-function}
		\subsection{Background and preliminaries}\label{sec:background}%
	Fixed-points of \eqref{eq:pg} are zeros of the \emph{natural residual}
	\begin{equation}
		\Rnat{\gamma}(x) := \gamma^{-1} (x - T_\gamma(x)).
	\end{equation}
	Methods like \panoc{} \cite{stella2017simple} and \zerofpr{} \cite{themelis2018forward} find such zeros through a linesearch procedure that uses the FBE \cite{patrinos2013proximal}
	\begin{equation*}
		\fbe(x) = \inf_{u \in \R^n} \{f(x) + \langle \nabla f(x), u - x \rangle + g(u) + \tfrac{1}{2\gamma} \norm{ u - x }^2\}
	\end{equation*}
	as a merit function, which can be explicitly expressed as
	\[
		\thinmuskip=.05\thinmuskip
		\medmuskip=.05\medmuskip
		\thickmuskip=.05\thickmuskip
		\fbe(x) = f(x) + g(T_\gamma(x)) - \gamma \langle \nabla f(x), \Rnat{\gamma}(x) \rangle + \tfrac{\gamma}{2} \norm{ \Rnat{\gamma}(x)}^2
	\]
	in terms of the natural residual \cite[Eq.\,(2.2)]{stella2017forward}.
	In particular, at every iteration, \zerofpr{}
	\begin{itemize}
		\item computes $\bar x = T_\gamma(x)$ and a `fast' direction $d \in \R^n$;
		\item constructs candidate points $x_{\rm cand}(\tau) \eqdef \bar x + \tau d$ with backtracking parameter $\tau \in (0, 1]$; and
		\item selects $x^+ = x_{\rm cand}(\tau)$ with $\tau$ such that, for some $\sigma > 0$, $\fbe(x_{\rm cand}(\tau)) - \fbe(x) \leq - \sigma \norm{ \Rnat{\gamma}(x) }^2$.
	\end{itemize}
	Three properties are crucial in ensuring well-definedness:
	\begin{enumerate}[label=(\roman*)]
		\item the point $\bar x = T_\gamma(x)$ yields \emph{sufficient decrease} on $\fbe$, specifically, $\fbe(\bar x) - \fbe(x) \leq - \gamma \frac{1 - \gamma L_f}{2} \norm{ \Rnat{\gamma}(x) }^2$;
		\item candidate points eventually approach this sufficient decrease point, i.e., $x_{\rm cand}(\tau) \to \bar x$ as $\tau \to 0$; and
		\item the FBE $\fbe$ is continuous, so that $\fbe(x_{\rm cand}(\tau)) \to \fbe(\bar x)$ as $\tau \to 0$.
	\end{enumerate}
	Properties (i) and (iii) are inherently related to the FBE, whereas (ii) follows by construction of the candidate points.
	\Panoc{} is similar in nature to \zerofpr{}, but it considers candidate points of the form $x_{\rm cand}(\tau) = (1-\tau) \bar x + \tau (x + d)$.
	While the algorithms are conceptually equivalent, this minor modification yields a cleaner update that saves one call to \(T_\gamma\) per iteration (albeit without necessarily reducing the total number of such calls compared to \zerofpr).
	We build specifically on the update rule of \zerofpr{}, as its structure is the one that enables fast asymptotic convergence in our setting.
	We nonetheless adopt the name `\panocpp*', as \panoc{} is more widely recognized and, as just noted, the two methods are conceptually equivalent.

	Other works \cite{ouyang2025trust,zeng2026linesearch} use the \emph{normal residual} \cite{robinson1992normal}
	\begin{equation}
		\Rnor{\gamma}(z) := \nabla f(\prox_{\gamma g} (z)) + \gamma^{-1} (z - \prox_{\gamma g}(z)),
	\end{equation}
	and propose algorithms involving the related merit function
	\begin{equation*}
		H_\tau(x) \eqdef \varphi(\prox_{\gamma g}(x)) + \tfrac{\tau \gamma}{2} \norm{ \Rnor{\gamma}(x) }^2 \text{ with }\tau \in (0, 1).
	\end{equation*}
	We emphasize that both the natural and normal residual characterize the first-order optimality conditions of \eqref{eq:problem-statement}.

	\begin{proposition}[{\cite[Lem.\,2.1]{ouyang2025trust}}]\label{prop:residuals}%
		Under \cref{assumption:basic},
		\begin{equation*}
			\prox_{\gamma g}(\zer \Rnor{\gamma}) \subseteq \zer \Rnat{\gamma} = \zer \partial \varphi = \zer(\nabla f + \partial g).
		\end{equation*}
	\end{proposition}

	Since merit functions are evaluated at every candidate point, their computational cost may significantly impact the overall efficiency of the method.
	The evaluation of \(\fbe\) and \(H_\tau\) requires the computation of \(\Rnat{\gamma}\) or \(\Rnor{\gamma}\), respectively, and hence the computation of \(\nabla f\) and \(\prox_{\gamma g}\).
	The following section proposes an alternative merit function that avoids these gradients $\nabla f$, thus considerably reducing the computational cost of an individual backtracking step.

\subsection{A novel merit function and its properties}
	Inspired by \(H_\tau\), we propose the merit function $\psi_\gamma := \varphi \circ \prox_{\gamma g}$, which notably does not involve gradients \(\nabla f\).
	We now present a sufficient decrease condition involving $\psi_\gamma$ and establish (strict) continuity of $\psi_\gamma$, thus mirroring similar properties of $\fbe$ that have proven crucial in ensuring well-definedness of \Panoc{} and \zerofpr{}.

	First, recall the cost decrease by a single PG iteration.

	\begin{lemma}[Cost decrease {\cite[Rem. 4(iii)]{bolte2014proximal}}]\label{lem:cost-decrease}%
		If \cref{assumption:basic} holds, then for any $\gamma > 0$ and any \( x \in \R^n \) we have
		\begin{equation}
			\varphi(T_\gamma(x)) - \varphi(x) \leq - \gamma \tfrac{2 - \gamma L_f}{2} \norm{ \Rnat{\gamma}(x) }^2.
		\end{equation}
	\end{lemma}

	Applying \cref{lem:cost-decrease} to a point \( \hat x = \prox_{\gamma g}(x) \) immediately yields a sufficient decrease condition for \( \psi_\gamma \), i.e.,
	\begin{equation} \label{eq:psi-decrease}
		\psi_\gamma(\hat x - \gamma \nabla f(\hat x)) - \psi_\gamma(x) \leq - \gamma \tfrac{2 - \gamma L_f}{2} \norm{ \Rnat{\gamma}(\hat x) }^2.
	\end{equation}

	Second, we establish strict continuity of \(\psi_\gamma\).

	\begin{lemma}\label{lem:psi-Lipschitz}%
		If \cref{assumption:basic-g} holds, then
		$g \circ \prox_{\gamma g}$ is strictly continuous for any $\gamma > 0$, with
		\begin{equation*}
			\lip[ g \circ \prox_{\gamma g}](x) \leq \tfrac{1}{\gamma} \norm{ x - \prox_{\gamma g}(x) }.
		\end{equation*}
	\end{lemma}
	\begin{proof}
		For \(x_i\in\R^n\) let \(\bar x_i\coloneqq\prox_{\gamma g}(x_i)\), \(i=1,2\).
		Then,
		\begin{align*}
		&
			g\circ\prox_{\gamma g}(x_1)
			-
			g\circ\prox_{\gamma g}(x_2)
		\\
		={} &
			g^\gamma(x_1)
			-
			g^\gamma(x_2)
			-
			\tfrac{1}{2\gamma}
			\norm{\bar x_1-x_1}^2
			+
			\tfrac{1}{2\gamma}
			\norm{\bar x_2-x_2}^2\\
		\leq{} &
			\langle\nabla g^\gamma(x_2),x_1-x_2\rangle
			+
			\tfrac{1}{2\gamma}
			\norm{x_1-x_2}^2
		\\
		&
			-
			\tfrac{1}{2\gamma}
			\norm{\bar x_1-x_1}^2
			+
			\tfrac{1}{2\gamma}
			\norm{\bar x_2-x_2}^2
		\\
		={} &
			\tfrac{1}{\gamma}
			\langle x_2-\bar x_2,x_1-x_2\rangle
			+
			\tfrac{1}{2\gamma}
			\norm{x_1-x_2}^2
		\\
		&
			-
			\tfrac{1}{2\gamma}
			\norm{\bar x_1-x_1}^2
			+
			\tfrac{1}{2\gamma}
			\norm{\bar x_2-x_2}^2
		\\
		={} &
			\tfrac{1}{2\gamma}
			\norm{\bar x_2-x_1}^2
			-
			\tfrac{1}{2\gamma}
			\norm{\bar x_1-x_1}^2
		\\
		\leq{} &
			\tfrac{1+\frac{1}{\varepsilon}}{2\gamma}
			\norm{\bar x_2-\bar x_1}^2
			+
			\tfrac{\varepsilon}{2\gamma}
			\norm{\bar x_1-x_1}^2
		\\
		\leq{} &
			\tfrac{1+\frac{1}{\varepsilon}}{2\gamma}
			\norm{x_2-x_1}^2
			+
			\tfrac{\varepsilon}{2\gamma}
			\norm{\bar x_1-x_1}^2,
		\end{align*}
		where the consecutive inequalities follow by \(\frac{1}{\gamma}\)-Lipschitz differentiability of \(g^\gamma\), Young's inequality, and nonexpansiveness of \(\prox_{\gamma g}\), respectively.
		With \(\varepsilon=\frac{\norm{x_1-x_2}}{\norm{x_1-\bar x_1}}\) we obtain
		\[
			\tfrac{
				g\circ\prox_{\gamma g}(x_1)
				-
				g\circ\prox_{\gamma g}(x_2)
			}{
				\norm{x_2-x_1}
			}
		\leq
			\tfrac{1}{2\gamma}
			\norm{x_2-x_1}
			+
			\tfrac{1}{\gamma}
			\norm{\bar x_1-x_1},
		\]
		and letting \(x_2,x_1\to x\) yields the claimed bound.
	\end{proof}

	\begin{corollary}\label{lem:psi-gamma-Lipschitz}%
		If \cref{assumption:basic-g} holds and if $f : \R^n \to \R$ is $\C^1$ on $\R^n$, then $\psi_\gamma := \varphi \circ \prox_{\gamma g}$ is strictly continuous for any $\gamma > 0$, with
		\begin{equation*}
			\lip \psi_\gamma(x) \leq \norm{\nabla f(\prox_{\gamma g}(x))} + \tfrac{1}{\gamma}\norm{x - \prox_{\gamma g}(x)}.
		\end{equation*}
	\end{corollary}

	\section{Algorithm}
		\subsection{Overview}\label{sec:overview}%
	We now introduce \cref{alg:zerofprpp}, a variant of \zerofpr{} \cite{themelis2018forward} based on the novel merit function $\psi_\gamma$.
	Recall that for a given iterate $x^k$, \zerofpr{} considers candidate points of the form
	\begin{equation*}
		x_{\rm cand}^{k+1}(\tau) \eqdef \bar x^k + \tau d^k, \qquad \bar x^k = T_\gamma(x^k),
	\end{equation*}
	where \(\tau > 0\) is a backtracking parameter.
	A candidate point $x_{\rm cand}^{k+1}(\tau)$ is accepted by the linesearch procedure whenever a sufficient decrease condition on the FBE is satisfied.
	The sufficient decrease condition ensures global convergence, and the proximal-gradient step and continuity of the FBE guarantee that the condition holds for sufficiently small values of \(\tau\).
	\Cref{alg:zerofprpp} follows a similar strategy, but instead verifies descent directly on the objective \(\varphi\).
	The key idea is that \(\varphi\) is evaluated \emph{only at points of the form \(\prox_{\gamma g}(x)\)}.
	As such, we implicitly work with the merit function \(\psi_\gamma\), and the linesearch condition \eqref{eq:linesearch-zerofprpp} becomes equivalent to
	\begin{equation}\label{eq:linesearch-zerofprpp-psi}
		\psi_\gamma(x^{k+1}) - \psi_\gamma(x^k) \leq - \sigma \norm{ \Rnat{\gamma}(\hat x^k) }^2.
	\end{equation}
	By leveraging the previously derived sufficient decrease condition \eqref{eq:psi-decrease} and the continuity of $\psi_\gamma$ (\cref{lem:psi-gamma-Lipschitz}), \cref{sec:convergence} establishes well-definedness and global convergence using similar arguments as for \zerofpr{} (\cref{thm:zerofpr:subseq}).
	\begin{algorithm}[H]
		\small
		\caption{\panocpp*}%
		\label{alg:zerofprpp}%
		\begin{algorithmic}[1]
			\Require{$x^0 \in \R^n, \gamma \in (0, \nicefrac{2}{L_f}), \beta \in (0, 1), \sigma \in (0, \gamma \frac{2 - \gamma L_f}{2})$;}
			\Initialize $\hat x^0 = \prox_{\gamma g}(x^0)$;
			\For{$k = 1, 2 \dots$}
				\State $\bar x^k = \hat x^k - \gamma \nabla f(\hat x^k)$;
				\State $\Rnat{\gamma}(\hat x^k) = \gamma^{-1} (\hat x^k - \prox_{\gamma g}(\bar x^k))$;
				\State
					Select a direction $\bar d^k \in \R^n$ at \( \bar x^k \) and let
					\begin{equation*}
						\begin{aligned}
							&x^{k+1} = \bar x^k + \tau_k \bar d^k\\
							&\hat x^{k+1}  =\prox_{\gamma g}(x^{k+1})
						\end{aligned}
					\end{equation*}
					\hspace*{\algorithmicindent}%
					where $\tau_k$ is the largest in $\left\{\beta^i \mid i \in \N \right\}$ such that
					\begin{equation} \label{eq:linesearch-zerofprpp}
						\varphi(\hat x^{k+1}) - \varphi(\hat x^k) \leq - \sigma \norm{ \Rnat{\gamma}(\hat x^k) }^2.
					\end{equation}
			\EndFor
		\end{algorithmic}
	\end{algorithm}

	Sufficiently close to a strong local minimum, well-chosen directions \(\bar d^k\) ensure that the unit stepsize \(\tau_k = 1\) is accepted, in which case \(x^{k+1} = \bar x^k + \bar d^k\).
	In particular, \cref{sec:convergence} formally proves a local superlinear convergence rate under a mild differentiability assumption on $\Rnor{\gamma}$ at the limit point $\bar x^\star$ (which we will see is equivalent to a similar assumption used by \zerofpr{}), and under a classical Dennis-Mor\'e condition
	\begin{equation} \label{eq:dennis-more}
		\lim_{k \to \infty} \frac{\norm{ \Rnor{\gamma}(\bar x^k) + \jac{\Rnor{\gamma}}(\bar x^\star) \bar d^k }}{\norm{ \bar d^k }} = 0.
	\end{equation}
	This motivates Newton-type directions of the form
	\begin{equation} \label{eq:newton-direction}
		\bar d^k = - H_k \Rnor{\gamma}(\bar x^k),
	\end{equation}
	where $H_k$ is an invertible operator that, ideally, captures curvature information of $\Rnor{\gamma}$.
	Practical quasi-Newton schemes update a \emph{linear} operator $H_k$ recursively in such a way that an inverse secant condition holds:
	\begin{equation*}
		\bar x^{k+1} - \bar x^k = H_{k+1} \left( \Rnor{\gamma}(\bar x^{k+1}) - \Rnor{\gamma}(\bar x^k) \right).
	\end{equation*}

	We emphasize that the combination of natural residuals (in the linesearch condition) and normal residuals (for the directions \(\bar d^k\)) is a distinctive feature of \cref{alg:zerofprpp}.

	\begin{remark}[Cheaper backtracking]
		At every iteration, \cref{alg:zerofprpp} requires one gradient \(\nabla f\) to compute \(\bar x^k\), and another gradient \(\nabla f\) to compute \(\Rnor{\gamma}(\bar x^k)\), which is needed for Newton-type directions \eqref{eq:newton-direction}.
		Hence, \cref{alg:zerofprpp} needs exactly \emph{two} gradients per iteration, regardless of the number of backtracking steps.
		In contrast, \zerofpr{} and \panoc{} perform an additional evaluation of \(\nabla f\) every time the linesearch procedure backtracks, as the FBE is re-evaluated.
	\end{remark}

	\begin{remark}[Larger stepsizes]\label{rem:larger-stepsizes}%
		Global convergence guarantees for \cref{alg:zerofprpp} hold if the stepsize $\gamma$ satisfies \(\gamma \in (0, \nicefrac{2}{L_f})\).
		This notably improves upon FBE-based methods which instead require smaller stepsizes \(\gamma \in (0, \nicefrac{1}{L_f})\).
	\end{remark}

	\begin{remark}[Adaptive variant]\label{rem:adaptive}%
		An unknown Lipschitz constant $L_f$ can be estimated adaptively using a similar procedure as described in \cite[Rem. 5.2]{themelis2018forward}.
		Fix a ratio \(\alpha \in (0, 1)\), an initial estimate $L > 0$, stepsize $\gamma \in (0, \nicefrac{2}{L})$, and parameter \(\sigma \in (0, \gamma \frac{2 - \gamma L}{2})\).
		At every iteration, after computing $\Rnat{\gamma}(\hat x^k)$, verify the Lipschitz upper bound between the points \(\hat x^k\) and $y^k \eqdef T_\gamma(\hat x^k)$.
		If violated, i.e., if
		\begin{equation*}
			f(y^k) > f(\hat x^k) - \langle \nabla f(\hat x^k), \hat x^k - y^k \rangle + \tfrac{L}{2} \norm{ \hat x^k - y^k }^2,
		\end{equation*}
		then update $\gamma \leftarrow \alpha \gamma$, $L \leftarrow \nicefrac{L}{\alpha}$, and $\sigma \leftarrow \alpha \sigma$ and restart that iteration.
		This can only happen a finite number of times, after which $L \geq L_f$.
		From then on $\gamma$ and $\sigma$ are also constant, and the sufficient decrease derived from \cref{lem:cost-decrease} still holds.
		In particular, all convergence results remain valid.
	\end{remark}

\subsection{Convergence analysis}\label{sec:convergence}%
	We now establish the well-definedness of \panocpp{} and use the sufficient decrease condition \eqref{eq:linesearch-zerofprpp} to show convergence.

	\begin{theorem}[Well-definedness]\label{thm:zerofpr:well-defined}%
		Suppose that \cref{assumption:basic} holds.
		Then \panocpp{} is well defined: every iteration terminates after a finite number of backtracking steps.
	\end{theorem}
	\begin{proof}
		We know from \eqref{eq:psi-decrease} that
		\(
			\psi_\gamma(\bar x^k) - \psi_\gamma(x^k) \leq - \gamma \tfrac{2 - \gamma L_f}{2} \norm{ \Rnat{\gamma}(\hat x^k) }^2.
		\) 
		Since \(x^{k+1}\to\bar x^k\) as \(\tau_k\to0\), and due to continuity of \(\psi_\gamma\) (\cref{lem:psi-gamma-Lipschitz}) and the fact that \(\sigma \in (0, \gamma \tfrac{2 - \gamma L_f}{2}) \), \eqref{eq:linesearch-zerofprpp-psi} holds for \(\tau_k\) small enough.
	\end{proof}

	\begin{theorem}[Global subsequential convergence]\label{thm:zerofpr:subseq}%
		Under \cref{assumption:basic}, a sequence \(\hat x^k\) generated by \panocpp{} satisfies:
		\begin{lemenum}
		\item
			\(\seq{\norm{\Rnat{\gamma}(\hat x^k)}^2}\) has finite sum;
		\item
			any limit point \(\hat x^\star\) of \(\seq{\hat x^k}\) satisfies \(\hat x^\star \in \zer \Rnat{\gamma}\);
		\item
			if \(\varphi\) is level bounded, then \(\seq{\hat x^k}\) is bounded.
		\end{lemenum}
	\end{theorem}
	\begin{proof}
		A standard telescoping argument and $\inf \varphi > -\infty$ establish that \(\sum_{k \in \N} \norm{ \Rnat{\gamma}(\hat x^k)}^2 < \infty\), and in particular \(\Rnat{\gamma}(\hat x^k) \to 0\).
		Therefore, any accumulation point \(\hat x^\star\) of \(\seq{\hat x^k}\) satisfies \(\hat x^\star = T_\gamma(\hat x^\star)\) (cfr.\,\cref{prop:residuals}), and is thus stationary (see \cite{stella2017simple} for a similar argument).
		Iterative application of \eqref{eq:linesearch-zerofprpp} reveals that \(\varphi(\hat x^k)\leq\varphi(\hat x^0)\) for all \(k\in\N\); hence \(\seq{\hat x^k}\) is bounded whenever \(\varphi\) is level bounded.
	\end{proof}

	Moreover, if the directions $\seq{\bar d^k}$ are superlinear with respect to \(\seq{\bar x^k}\), then unit stepsize is eventually accepted, achieving a local superlinear rate of convergence.

	\begin{theorem}[Superlinear convergence]\label{thm:superlinear-convergence}%
		Let \cref{assumption:basic} hold.
		Consider the iterates generated by \panocpp{} and suppose that \(\hat x^k\) converges to a strong local minimizer \(x^\star\) of \(\varphi\).
		Suppose further that the directions \(\seq{\bar d^k}\) are superlinear with respect to \(\seq{\bar x^k}\), in the sense that
		\begin{equation}\label{eq:dk}
			\lim_{k\to\infty}\frac{\norm{ \bar x^k+\bar d^k-\bar x^\star}}{\norm{ \bar x^k-\bar x^\star }} = 0
		\end{equation}
		with \(\bar x^\star:=x^\star - \gamma \nabla f(x^\star)\).
		Then \(\tau_k=1\) is eventually always accepted and \(\seq{\hat x^k}\) converges superlinearly.
	\end{theorem}
	\begin{proof}
		Up to discarding early iterates, there exists \(\mu>0\) such that \(\varphi(\hat x^k)-\varphi_\star\geq\tfrac{\mu}{2}\norm{ \hat x^k-x^\star}\) holds for every \(k\), where \(\varphi_\star:=\varphi(x^\star)\).
		It follows from \cref{thm:QUB} that
		\begin{equation*}
			\begin{aligned}
				\epsilon_k := \frac{
					\psi_\gamma(\bar x^k+\bar d^k)-\varphi_\star
				}{
					\psi_\gamma(x^k)-\varphi_\star
				}
			&\leq
				\frac{1+\gamma L_f}{\gamma\mu}
				\frac{
					\norm{ \bar x^k+\bar d^k-\bar x^\star}^2
				}{
					\norm{ \hat x^k-x^\star }^2
				}\\
			&\leq
				\frac{(1+\gamma L_f)^2}{\gamma\mu}
				\frac{
					\norm{ \bar x^k+\bar d^k-\bar x^\star}^2
				}{
					\norm{\bar x^k-\bar x^\star}^2
				},
			\end{aligned}
		\end{equation*}
		where the last inequality is due to the \((1+\gamma L_f)\)-Lipschitz continuity of \(\id - \gamma \nabla f\).
		As such, it follows that \(\epsilon_k\to0\) as \(k\to\infty\), and we can thus assume without loss of generality that \(\epsilon_k\leq1\) holds for all \(k\).
		Similarly, since \(y^k:=\prox_{\gamma g}(\hat x^k - \gamma \nabla f(\hat x^k))\) also converges to \(x^\star\), we may assume that for all \(k\) it holds that
		\(
			\varphi(y^k)
		\geq
			\varphi_\star
		\).
		Then,
		\begin{align*}
			\psi_\gamma(\bar x^k+\bar d^k)-\psi_\gamma(x^k)
		={} &
			-(1-\epsilon_k)
			\bigl[\psi_\gamma(x^k)-\varphi_\star\bigr]
		\\
		\leq{} &
			-(1-\epsilon_k)
			\bigl[\psi_\gamma(x^k)-\varphi(y^k)\bigr]
		\\
		\leq{} &
			-(1-\epsilon_k)
			\gamma\tfrac{2-\gamma L_f}{2}
			\norm{ \Rnat{\gamma}(\hat x^k) }^2
		\end{align*}
		is eventually smaller than \(-\sigma\norm{ \Rnat{\gamma}(\hat x^k) }^2\), resulting in acceptance of \(\tau_k=1\).
		It follows that the \(\bar x^k\)-update eventually boils down to
		\[
			\begin{cases}
				\hat x^{k+1} = \prox_{\gamma g}(\bar x^k+\bar d^k)
			\\
				\bar x^{k+1} = \hat x^{k+1} - \gamma \nabla f(\hat x^{k+1}),
			\end{cases}
		\]
		hence, by nonexpansiveness of \(\prox_{\gamma g}\) and \((1+\gamma L_f)\)-Lipschitz continuity of \(\id - \gamma \nabla f\),
		\[
			\lim_{k\to\infty}
			\frac{
				\norm{ \hat x^{k+1}-x^\star }
			}{
				\norm{ \hat x^k-x^\star }
			}
		\leq
			\lim_{k\to\infty}
			\frac{
				\norm{ \bar x^k+\bar d^k-\bar x^\star }
			}{
				\frac{1}{1+\gamma L_f}
				\norm{ \bar x^k-\bar x^\star }
			}
		=
			0.
		\qedhere
		\]
	\end{proof}

	If $\Rnor{\gamma}$ is strictly differentiable at the limit $\bar x^\star$, then Newton-type directions of the form \eqref{eq:newton-direction} are superlinear directions \eqref{eq:dk} under a Dennis-Mor\'e condition \eqref{eq:dennis-more}.
	Before proving this, we first describe the local differentiability of the normal residual under the same assumption that yields differentiability of the natural residual \cite[Assumpt. II]{themelis2018forward}.
	Regarding the notion of \emph{epi-differentiability}, we refer the interested reader to \cite[\S 13]{rockafellar1998variational}.

	\begin{assumption}\label{ass:second-order}%
		With respect to a given point \(x^\star \in \fix T_\gamma\)
		\begin{assumenum}
		\item
			\(\nabla^2f\) exists and is (strictly) continuous around \(x^\star\);
		\item
			\(g\) is (strictly) twice epi-differentiable at \(x^\star\) for \(-\nabla f(x^\star)\), with generalized quadratic second order epi-derivative.%
		\end{assumenum}
		The assumptions are said to be ``strictly'' satisfied whenever the stronger conditions in parenthesis hold.
	\end{assumption}

	Twice epi-differentiability is a necessary and sufficient condition for differentiability of $\prox_{\gamma g}$ at \(x^\star - \gamma \nabla f(x^\star)\) \cite{poliquin1996generalized}.
	This is a mild assumption, as exemplified by the fact that it holds for the broad class of $\C^2$-partly smooth functions \cite{lewis2002active} under a mild strict complementarity condition $-\nabla f(x^\star) \in \relint \partial g(x^\star)$ \cite[Thm.\,28]{daniilidis2006geometrical}, see also \cite{bodard2025second}.

	\begin{theorem}[Differentiability of the normal residual]\label{thm:residual-diff}%
		Let \cref{assumption:basic} hold and let \(\gamma > 0\).
		Suppose further that \cref{ass:second-order} holds (strictly) with respect to a point \(x^\star \in \fix T_\gamma\).
		Then, the normal residual \(\Rnor{\gamma}\) is (strictly) differentiable at \(\bar x^\star := x^\star - \gamma \nabla f(x^\star)\) with Jacobian
		\[
			\jac{\Rnor{\gamma}} (\bar x^\star) = \gamma^{-1} \left[ \id - Q_\gamma(x^\star) P_\gamma(x^\star) \right]
		\]
		where \( Q_\gamma(x^\star) := \id - \gamma \nabla^2 f(x^\star) \), \(P_\gamma(x^\star) := \jac{\prox_{\gamma g}}(\bar x^\star) \).
	\end{theorem}
	\begin{proof}
		Under \cref{ass:second-order}, \(\prox_{\gamma g}\) is (strictly) differentiable at \( x^\star - \gamma \nabla f(x^\star) \eqqcolon \bar x^\star \) \cite{poliquin1996generalized}, from which we obtain
		\begin{equation*}
			\begin{aligned}
				\jac{\Rnor{\gamma}}(\bar x^\star) ={} &\nabla^2 f(\prox_{\gamma g}(\bar x^\star)) \jac{\prox_{\gamma g}}(\bar x^\star)\\
				&+ \gamma^{-1} (\id - \jac{\prox_{\gamma g}}(\bar x^\star)).
			\end{aligned}
		\end{equation*}
		Using the fact that \( P_\gamma(x^\star) = \jac{\prox_{\gamma g}}(\bar x^\star) \) and grouping terms involving \(P_\gamma(x^\star)\) yields
		\[
			\jac{\Rnor{\gamma}}(\bar x^\star) = \gamma^{-1} \left( \id - Q_\gamma(\prox_{\gamma g}(\bar x^\star)) P_\gamma(x^\star) \right).
		\]
		The claim now follows by observing that \(\prox_{\gamma g}(\bar x^\star) = \prox_{\gamma g}(x^\star - \gamma \nabla f(x^\star)) = x^\star\).
	\end{proof}

	Observe from \cite[Thm. 4.10]{themelis2018forward} that the Jacobians of \(\Rnat{\gamma}\) and \(\Rnor{\gamma}\) are related by the relation
	\begin{equation}\label{eq:JR}
		\jac{\Rnat{\gamma}}(x^\star) = \jac{\Rnor{\gamma}}(\bar x^\star)^\top.
	\end{equation}
	Hence, results involving the natural residual are easily transferred to the normal residual.

	\begin{corollary}[Conditions for strong local minimality]\label{corr:strong-local-minimality}%
		Let \cref{assumption:basic} hold and let \(\gamma\in(0,\nicefrac{1}{L_f})\).
		Suppose further that \cref{ass:second-order} is satisfied with respect to a point \(x^\star \in \fix T_\gamma\).
		Then the following are equivalent:
		\begin{corollenum}
			\item \(x^\star\) is a strong local minimizer for \(\varphi\);
			\item \(x^\star\) is a local minimizer for \(\varphi\) and \( \jac{\Rnat{\gamma}}(x^\star)\) is nonsingular;
			\item \(x^\star\) is a local minimizer for \(\varphi\) and \( \jac{\Rnor{\gamma}}(\bar x^\star)\) with \(\bar x^\star := x^\star - \gamma \nabla f(x^\star)\) is nonsingular.
		\end{corollenum}
	\end{corollary}
	\begin{proof}
		The equivalence between the first two items is shown in \cite[Thm. 4.11]{themelis2018forward}; in turn, that between the last two items follows from \eqref{eq:JR}.
	\end{proof}

	When $\jac{\Rnor{\gamma}}(\bar x^\star)$ is nonsingular and the Dennis-Mor\'e condition \eqref{eq:dennis-more} holds, the directions $\seq{\bar d^k}$ are superlinear, which by \cref{thm:superlinear-convergence} implies superlinear convergence of $\seq{\hat x^k}$.

	\begin{theorem}[Superlinear convergence under Dennis-Mor\'e condition]
		Let \cref{assumption:basic} hold, and let \(\gamma \in (0, \nicefrac{1}{L_f})\).
		Consider the iterates generated by \cref{alg:zerofprpp} and suppose that \cref{ass:second-order} is strictly satisfied at a strong local minimizer \(x^\star\) of \(\varphi\).
		Moreover, suppose that \( \left( \hat x^k \right)_{k \in \N} \) converges to \( x^\star \) and that the directions \(\left(\bar d^k\right)_{k \in \N}\) satisfy the Dennis-Mor\'e condition \eqref{eq:dennis-more}
		where
		\(
			\bar x^\star \coloneqq x^\star - \gamma \nabla f(x^\star).
		\)
		Then, \eqref{eq:dk} holds and in particular all the claims from \cref{thm:superlinear-convergence} hold.
	\end{theorem}
	\begin{proof}
		By continuity of \(\nabla f\), convergence of \( \left( \hat x^k \right)_{k \in \N} \) to \(x^\star\) implies convergence of \(\left(\bar x^k\right)_{k \in \N}\) to \( \bar x^\star \).
		From $x^\star \in \fix T_\gamma$, we derive $\Rnor{\gamma}(\bar x^k) \to 0$ and by \eqref{eq:dennis-more} that $\bar d^k \to 0$.
		Note that
		\begin{equation*}
			\begin{aligned}
				&\frac{\Rnor{\gamma}(\bar x^k) + \jac{\Rnor{\gamma}}(\bar x^\star) \bar d^k}{\norm{ \bar d^k }} = \frac{\Rnor{\gamma}(\bar x^k + \bar d^k)}{\norm{ \bar d^k }}\\
				&\qquad+ \frac{\Rnor{\gamma}(\bar x^k) + \jac{\Rnor{\gamma}}(\bar x^\star) \bar d^k - \Rnor{\gamma}(\bar x^k + \bar d^k)}{\norm{ \bar d^k }}
			\end{aligned}
		\end{equation*}
		Then, since $\bar d^k \to 0$ and from the Dennis-Mor\'e condition \eqref{eq:dennis-more} and strict differentiability of \(\Rnor{\gamma}\) at \(\bar x^\star\) it follows that
		\[
			\lim_{k \to \infty} \frac{\norm{ \Rnor{\gamma}(\bar x^k + \bar d^k) }}{\norm{ \bar d^k }} = 0.
		\]
		Nonsingularity of \(\jac{\Rnor{\gamma}}(\bar x^\star)\) (from \cref{corr:strong-local-minimality}) ensures that for all \(\bar x\) sufficiently close to \(\bar x^\star\) there exists a constant \(\alpha>0\) such that \(\norm{ \Rnor{\gamma}(\bar x)} \geq \alpha \norm{ \bar x - \bar x^\star }\).
		Since \(d^k \to 0\) by the Dennis-Mor\'e condition, also \(\bar x^k + \bar d^k \to \bar x^\star \), and eventually this lower bound must hold.
		Therefore,
		\begin{equation*}
			\begin{aligned}
			0 &\leftarrow \frac{\norm{ \Rnor{\gamma}(\bar x^k + \bar d^k)}}{\norm{ \bar d^k }} \geq \alpha \frac{\norm{ \bar x^k + \bar d^k - \bar x^\star}}{\norm{ \bar d^k}}\\
			&\geq \alpha \frac{\norm{ \bar x^k + \bar d^k - \bar x^\star}}{\norm{ \bar x^k + \bar d^k - \bar x^\star} + \norm{\bar x^k - \bar x^\star}} = \alpha \frac{\frac{\norm{ \bar x^k + \bar d^k - \bar x^\star}}{\norm{ \bar x^k - \bar x^\star}}}{1+\frac{\norm{ \bar x^k + \bar d^k - \bar x^\star}}{\norm{ \bar x^k - \bar x^\star}}}
			\end{aligned}
		\end{equation*}
		and hence \eqref{eq:dk} follows.
	\end{proof}

	We conclude with a remark on \emph{structured} directions \cite{pas2022alpaqa}.

	\begin{remark}[Structured directions]
		Update directions leveraging additional structure of $g$ were proposed in \cite[\S III]{pas2022alpaqa}.
		If $g$ is the indicator of a box, then define index sets of \emph{active} (\(\cK\)) and inactive (\(\cJ\)) constraints at $\hat x - \gamma \nabla f(\hat x)$.
		Since the Jacobian (if it exist) of $\prox_{\gamma g}$ at this forward point is a diagonal matrix with ones and zeros,
		Newton directions $d = -\jac{\Rnat{\gamma}}(\hat x) \Rnat{\gamma}(\hat x)$ satisfy \(\frac{1}{\gamma} d_{\cK} = - \left[ \Rnat{\gamma}(\hat x) \right]_{\cK}\) and
		\begin{equation*}
			\nabla_{\!\cJ\!\cJ}^2 f(\hat x) d_{\cJ} = - \nabla_{\!\cJ} f(\hat x) - \nabla_{\!\cJ\mkern-2mu\cK}^2 f(\hat x) d_{\cK}.
		\end{equation*}
		Thus, $d_{\cJ}$ is `close' to a Newton direction of $f$ restricted to inactive constraints.
		A similar derivation holds for Newton directions $d = -\jac{\Rnor{\gamma}}(x) \Rnor{\gamma}(x)$.
		By a slight abuse of notation, let $\cK$ ($\cJ$) now denote the index sets of active (inactive) constraints at $x$ and define \(\hat x = \prox_{\gamma g}(x)\). Then,
		\begin{equation*}
			\begin{aligned}
				\tfrac{1}{\gamma} d_{\cK} &= - \left[ \Rnor{\gamma}(x) \right]_{\cK} - \nabla_{\cK\mkern-1mu\cJ}^2 f(\hat x) d_{\cJ}\\
				\nabla_{\!\cJ\!\cJ}^2 f(\hat x) d_{\cJ} &= - \nabla_{\!\cJ} f(\hat x).
			\end{aligned}
		\end{equation*}
		In this case $d_{\cJ}$ exactly reduces to a Newton direction of $f$ restricted to the inactive constraints.
		A similar structure is also present if $g = \norm{ \cdot }_1$.
	\end{remark}

	\section{Numerical results}
		We demonstrate the effectiveness of \cref{alg:zerofprpp} against the two state-of-the-art solvers it builds upon: \zerofpr{} \cite{themelis2018forward} and \panoc{} \cite{stella2017forward}.
Experiments include two nonlinear model predictive control (MPC) problems with challenging state constraints, as well as the \textsc{libsvm} and \textsc{Cute}st benchmarks.\footnote{The source code to reproduce the results in this section can be found at
\href{https://github.com/alexanderbodard/panoc-lite-experiments}{\texttt{github.com/alexanderbodard/panoc-lite-experiments}}.}

\subsection{Model predictive control for collision avoidance}\label{sec:mpc}%
	In this experiment we simulate MPC controllers with prediction horizon $N = 32$ for two problems with collision avoidance constraints.
	The first involves the simplified quadcopter model and cylindrical object from \cite[\S IV]{bodard2023pantr}.
	The second models a kinematic bicycle model and an S-shaped corridor, as described in \cite[\S 4.2]{hermans2021penalty}.
	We use the same problem parameters as mentioned in these references.
	Optimal trajectories for both systems are visualized in \cref{fig:mpc-trajectories}.
	\begin{figure}[h]
		\centering
		\includegraphics[scale=0.48, clip, trim=0cm 0.75cm 0cm 0.39cm]{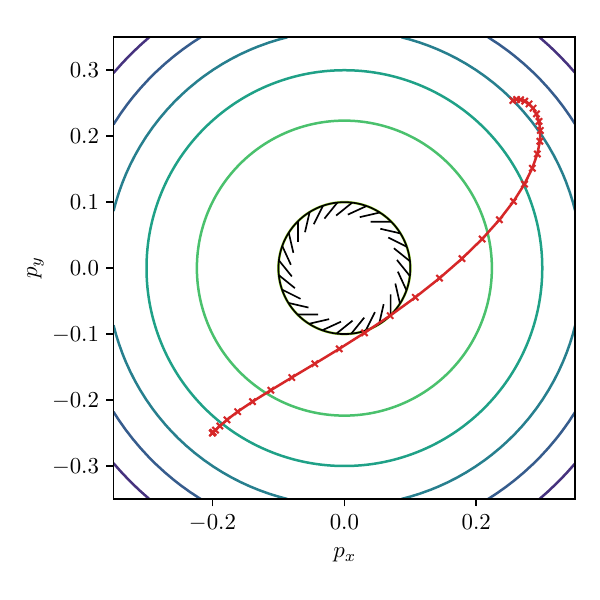}
		\includegraphics[scale=0.459168, clip, trim=1.65cm 0.445cm 2cm 0.39cm]{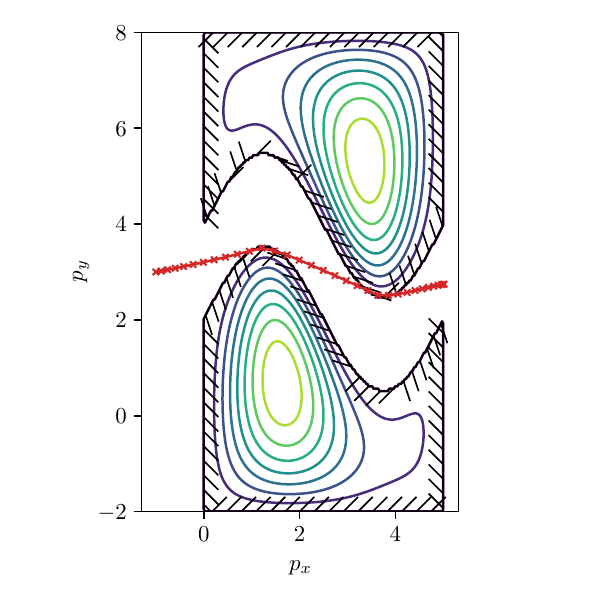}
		\vspace{-0.75em}
		\caption{Optimal trajectories of the quadcopter (left) and bicycle (right).}
		\label{fig:mpc-trajectories}
	\end{figure}

	MPC controllers solve optimal control problems (OCPs)
	\begin{equation} \label{eq:ocp}
		\begin{aligned}
			&  & \minimize_{\bm z, \bm u} \;\; & \sum_{k=0}^{N-1} \ell_k(z^k, u^k) + \ell_N(z^N)                                                              \\
			&  & \stt\;\;                               & \begin{aligned}[t]
															& z^{k+1} = \Gamma(z^k, u^k),                   &  & \forall k \in \N_{[0, N-1]} \\
															& \underline u \le u^k \le \overline u,    &  & \forall k \in \N_{[0, N-1]} \\
															& \underline c \le c(z^k) \le \overline c, &  & \forall k \in \N_{[0, N].}
														\end{aligned}
		\end{aligned}
	\end{equation}
	Here the function $\Gamma : \R^\nnz \times \R^\nnu \to \R^\nnz$ defines the discrete-time dynamics of a system with states $z \in \R^\nnz$ and inputs $u \in \R^\nnu$.
	The loss is a sum of stage costs $\ell_k : \R^\nnz \times \R^\nnu \to \R$ and a terminal cost $\ell_N : \R^\nnz \to \R$.
	The control inputs $u^k$ are bounded by $\underline u, \overline u \in \R^\nnu$, and general state constraints are enforced through the nonlinear mapping $c : \R^\nnz \to \R^\nnc$ in combination with the bounds $\underline c, \overline c \in \R^\nnc$.
	Note that this OCP minimizes over sequences of states $\bm z \eqdef (z^1, \dots, z^N) \in \R^{\nnz \cdot N}$ and inputs $\bm u \eqdef (u^0, \dots u^{N-1}) \in \R^{\nnu \cdot N}$.
	Henceforth, we use instead a \emph{single-shooting} formulation of the OCP that eliminates the dynamics constraint from \eqref{eq:ocp}.
	Denoting $n = \nnu \cdot N$ and $m = \nnc \cdot N$, this yields an NLP of the form
	\begin{equation} \label{eq:nlp}
		\begin{aligned}
			&\minimize_{\bm u \in \R^n} & &\psi(\bm u)\\
			&\stt & &\bm u \in C, \quad h(\bm u) \in D,
		\end{aligned}
	\end{equation}
	where both \(\psi : \R^n \to \R\) and \(h : \R^n \to \R^m\) are potentially nonconvex functions, the sets $C$ and $D$ are closed and simple to project onto.
	The solver library \alpaqa{} \cite{pas2022alpaqa} addresses such problems with an ALM:
	it introduces penalties $\Sigma \in \R^{m \times m}$ and multipliers $y \in \R^m$, and reformulates \eqref{eq:nlp} into a sequence of (simpler) problems \eqref{eq:problem-statement}, where \cite[\S II]{pas2022alpaqa}
	\begin{equation*}
		f(\bm u) = \psi(\bm u) + \tfrac{1}{2} \dist_\Sigma^2(h(\bm u)+\Sigma^{-1} y, D), \quad g(\bm u) = \delta_C(\bm u).
	\end{equation*}

	For both problems we run \alpaqa{} with default configurations, and compare three different ALM inner solvers: \cref{alg:zerofprpp}, \zerofpr{} and \panoc{}.
	Convergence is declared when $\norm{ \bm u - \proj_{C}(\bm u - \nabla f(\bm{u}))}_\infty \leq 10^{-4}$ and $\norm{ c(\bm x) - \proj_{D}(c(\bm{x}) + \Sigma^{-1} y) }_\infty \leq 10^{-4}$.
	All inner solvers employ L-BFGS directions with a buffer size of $50$, and an estimated Lipschitz constant $L$ (cfr.\,\cref{rem:adaptive}), with \panoc{} using the improved linesearch from \cite{pas2022alpaqa,demarchi2022proximal}.
	\Cref{alg:zerofprpp} uses $\gamma = \nicefrac{1.95}{L}$, the others $\gamma = \nicefrac{0.95}{L}$ (cfr.\,\cref{rem:larger-stepsizes}).

	\Cref{fig:quadcopter,fig:bicycle} visualize the total number of function and gradient evaluations of $f$ per MPC time step for the quadcopter and bicycle problem, respectively, for both cold (left) and warm (right) starts.
	Remark how \cref{alg:zerofprpp} oftentimes significantly outperforms the other methods, especially during the most difficult initial MPC time steps.

	\begin{figure}[h]
		\centering
		\resizebox{\linewidth}{!}{

\definecolor{lightblue}{rgb}{0.00,0.6056,0.9787}
\definecolor{darkblue}{rgb}{0.00,0.40,0.70}

\definecolor{lightorange}{rgb}{0.8889,0.4356,0.2781}
\colorlet{darkorange}{orange!85!black}

\definecolor{lightgreen}{rgb}{0.65,0.85,0.70}
\definecolor{darkgreen}{rgb}{0.2422,0.6433,0.3044}

\begin{tikzpicture}

\begin{groupplot}[
    group style={
        group size=2 by 1,
        horizontal sep=1.4cm,
        ylabels at=edge left,
        xticklabels at=edge bottom,
    },
    width=0.47\textwidth,
    height=8.2cm,
    xmin=1, xmax=30,
    xtick={5,10,15,20,25,30},
    xlabel={MPC time step},
    ymode=log,
    ymin=1e2, ymax=7e4,
    grid=both,
    major grid style={draw=black!12},
    minor grid style={draw=black!6},
    tick align=outside,
    tick style={black!70},
    label style={font=\normalsize},
    ticklabel style={font=\small},
    title style={font=\normalsize},
    every axis plot/.append style={
        mark size=1.8pt,
        line width=1.1pt,
    },
]

\nextgroupplot[
    title={Cold start},
    ylabel={Function + gradient evaluations},
    legend to name=grouplegend,
    legend columns=3,
    legend style={
        draw,
        fill=white,
        fill opacity=0.9,
        text opacity=1,
        row sep = 1pt,
        column sep=8pt,
    },
]

\addplot[
    color=darkblue,
    mark=*,
]
table[row sep=\\]{
\\
            1.0  33900.0  \\
            2.0  32314.0  \\
            3.0  32398.0  \\
            4.0  30068.0  \\
            5.0  27137.0  \\
            6.0  26664.0  \\
            7.0  24684.0  \\
            8.0  24123.0  \\
            9.0  23945.0  \\
            10.0  17717.0  \\
            11.0  14923.0  \\
            12.0  9332.0  \\
            13.0  5286.0  \\
            14.0  5120.0  \\
            15.0  4841.0  \\
            16.0  4679.0  \\
            17.0  4327.0  \\
            18.0  3475.0  \\
            19.0  3814.0  \\
            20.0  3688.0  \\
            21.0  4183.0  \\
            22.0  4021.0  \\
            23.0  3847.0  \\
            24.0  3910.0  \\
            25.0  4043.0  \\
            26.0  3637.0  \\
            27.0  3877.0  \\
            28.0  3745.0  \\
            29.0  3850.0  \\
            30.0  3550.0  \\
};
\addlegendentry{\panoc{}}

\addplot[
    color=darkorange,
    mark=square*,
]
table[row sep=\\]{
\\
            1.0  20073.0  \\
            2.0  17821.0  \\
            3.0  17759.0  \\
            4.0  17037.0  \\
            5.0  16954.0  \\
            6.0  18662.0  \\
            7.0  15502.0  \\
            8.0  16391.0  \\
            9.0  13457.0  \\
            10.0  11084.0  \\
            11.0  10283.0  \\
            12.0  8885.0  \\
            13.0  7985.0  \\
            14.0  4893.0  \\
            15.0  4117.0  \\
            16.0  3797.0  \\
            17.0  4083.0  \\
            18.0  3907.0  \\
            19.0  3371.0  \\
            20.0  3781.0  \\
            21.0  3213.0  \\
            22.0  4109.0  \\
            23.0  3527.0  \\
            24.0  3453.0  \\
            25.0  4079.0  \\
            26.0  3277.0  \\
            27.0  3315.0  \\
            28.0  2907.0  \\
            29.0  3447.0  \\
            30.0  3609.0  \\
};
\addlegendentry{\zerofpr{}}

\addplot[
    color=darkgreen,
    mark=triangle*,
]
table[row sep=\\]{
\\
            1.0  14175.0  \\
            2.0  13158.0  \\
            3.0  13288.0  \\
            4.0  14184.0  \\
            5.0  13069.0  \\
            6.0  13364.0  \\
            7.0  13186.0  \\
            8.0  12254.0  \\
            9.0  13192.0  \\
            10.0  10657.0  \\
            11.0  9318.0  \\
            12.0  9711.0  \\
            13.0  6839.0  \\
            14.0  2691.0  \\
            15.0  2405.0  \\
            16.0  2598.0  \\
            17.0  2339.0  \\
            18.0  2284.0  \\
            19.0  2761.0  \\
            20.0  2425.0  \\
            21.0  2574.0  \\
            22.0  2445.0  \\
            23.0  2335.0  \\
            24.0  2523.0  \\
            25.0  2267.0  \\
            26.0  2335.0  \\
            27.0  2355.0  \\
            28.0  2546.0  \\
            29.0  2504.0  \\
            30.0  2596.0  \\
};
\addlegendentry{\zerofprpp{}}

\nextgroupplot[
    title={Warm start},
    yticklabels={},
]

\addplot[
    color=darkblue,
    mark=*,
]
table[row sep=\\]{
\\
            1.0  33900.0  \\
            2.0  45779.0  \\
            3.0  45074.0  \\
            4.0  34767.0  \\
            5.0  24624.0  \\
            6.0  1752.0  \\
            7.0  987.0  \\
            8.0  624.0  \\
            9.0  402.0  \\
            10.0  598.0  \\
            11.0  367.0  \\
            12.0  418.0  \\
            13.0  280.0  \\
            14.0  508.0  \\
            15.0  277.0  \\
            16.0  429.0  \\
            17.0  222.0  \\
            18.0  258.0  \\
            19.0  468.0  \\
            20.0  237.0  \\
            21.0  300.0  \\
            22.0  279.0  \\
            23.0  348.0  \\
            24.0  168.0  \\
            25.0  363.0  \\
            26.0  273.0  \\
            27.0  201.0  \\
            28.0  342.0  \\
            29.0  207.0  \\
            30.0  447.0  \\
};

\addplot[
    color=darkorange,
    mark=square*,
]
table[row sep=\\]{
\\
            1.0  20073.0  \\
            2.0  36651.0  \\
            3.0  41111.0  \\
            4.0  18311.0  \\
            5.0  14112.0  \\
            6.0  363.0  \\
            7.0  891.0  \\
            8.0  237.0  \\
            9.0  363.0  \\
            10.0  247.0  \\
            11.0  231.0  \\
            12.0  247.0  \\
            13.0  287.0  \\
            14.0  209.0  \\
            15.0  293.0  \\
            16.0  191.0  \\
            17.0  247.0  \\
            18.0  299.0  \\
            19.0  203.0  \\
            20.0  327.0  \\
            21.0  175.0  \\
            22.0  333.0  \\
            23.0  207.0  \\
            24.0  331.0  \\
            25.0  253.0  \\
            26.0  147.0  \\
            27.0  143.0  \\
            28.0  367.0  \\
            29.0  203.0  \\
            30.0  327.0  \\
};

\addplot[
    color=darkgreen,
    mark=triangle*,
]
table[row sep=\\]{
\\
            1.0  14175.0  \\
            2.0  24672.0  \\
            3.0  23647.0  \\
            4.0  13725.0  \\
            5.0  10575.0  \\
            6.0  469.0  \\
            7.0  1750.0  \\
            8.0  237.0  \\
            9.0  288.0  \\
            10.0  218.0  \\
            11.0  381.0  \\
            12.0  299.0  \\
            13.0  209.0  \\
            14.0  317.0  \\
            15.0  268.0  \\
            16.0  377.0  \\
            17.0  318.0  \\
            18.0  202.0  \\
            19.0  256.0  \\
            20.0  226.0  \\
            21.0  286.0  \\
            22.0  295.0  \\
            23.0  179.0  \\
            24.0  179.0  \\
            25.0  374.0  \\
            26.0  198.0  \\
            27.0  283.0  \\
            28.0  200.0  \\
            29.0  311.0  \\
            30.0  202.0  \\
};

\end{groupplot}

\node[anchor=south] at ($(group c1r1.north)!0.5!(group c2r1.north)+(0,0.8cm)$)
{\pgfplotslegendfromname{grouplegend}};

\end{tikzpicture}}
		\caption{Performance (function + gradient calls) per MPC time step for the quadcopter with horizon $N = 32$. Cold (left) and warm (right) starts.}
		\label{fig:quadcopter}
	\end{figure}
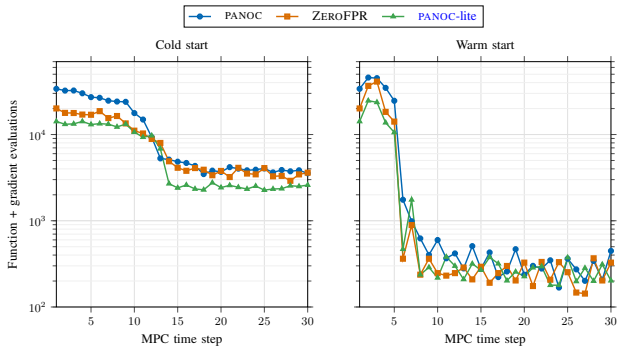

	\begin{figure}[h]
		\centering
		\resizebox{\linewidth}{!}{

\definecolor{lightblue}{rgb}{0.00,0.6056,0.9787}
\definecolor{darkblue}{rgb}{0.00,0.40,0.70}

\definecolor{lightorange}{rgb}{0.8889,0.4356,0.2781}
\colorlet{darkorange}{orange!85!black}

\definecolor{lightgreen}{rgb}{0.65,0.85,0.70}
\definecolor{darkgreen}{rgb}{0.2422,0.6433,0.3044}

\begin{tikzpicture}

\begin{groupplot}[
    group style={
        group size=2 by 1,
        horizontal sep=1.4cm,
        ylabels at=edge left,
        xticklabels at=edge bottom,
    },
    width=0.47\textwidth,
    height=8.2cm,
    xmin=1, xmax=30,
    xtick={5,10,15,20,25,30},
    xlabel={MPC time step},
    ymode=log,
    ymin=3e1, ymax=2e4,
    grid=both,
    major grid style={draw=black!12},
    minor grid style={draw=black!6},
    tick align=outside,
    tick style={black!70},
    label style={font=\normalsize},
    ticklabel style={font=\small},
    title style={font=\normalsize},
    every axis plot/.append style={
        mark size=1.8pt,
        line width=1.1pt,
    },
]

\nextgroupplot[
    title={Cold start},
    ylabel={Function + gradient evaluations},
    legend to name=grouplegendCS,
    legend columns=3,
    legend style={
        draw,
        fill=white,
        fill opacity=0.9,
        text opacity=1,
        row sep = 1pt,
        column sep=8pt,
    },
]

\addplot[
    color=darkblue,
    mark=*,
]
table[row sep=\\]{
\\
            1.0  11496.0  \\
            2.0  9661.0  \\
            3.0  5604.0  \\
            4.0  5974.0  \\
            5.0  5120.0  \\
            6.0  5626.0  \\
            7.0  4791.0  \\
            8.0  5726.0  \\
            9.0  4885.0  \\
            10.0  4224.0  \\
            11.0  4813.0  \\
            12.0  3472.0  \\
            13.0  3379.0  \\
            14.0  2976.0  \\
            15.0  2771.0  \\
            16.0  2027.0  \\
            17.0  1619.0  \\
            18.0  2216.0  \\
            19.0  2897.0  \\
            20.0  1977.0  \\
            21.0  1722.0  \\
            22.0  1236.0  \\
            23.0  1284.0  \\
            24.0  1220.0  \\
            25.0  1123.0  \\
            26.0  721.0  \\
            27.0  670.0  \\
            28.0  976.0  \\
            29.0  643.0  \\
            30.0  616.0  \\
};
\addlegendentry{\panoc{}}

\addplot[
    color=darkorange,
    mark=square*,
]
table[row sep=\\]{
\\
            1.0  8380.0  \\
            2.0  8040.0  \\
            3.0  4663.0  \\
            4.0  4599.0  \\
            5.0  7309.0  \\
            6.0  3601.0  \\
            7.0  3143.0  \\
            8.0  4373.0  \\
            9.0  2899.0  \\
            10.0  2961.0  \\
            11.0  2883.0  \\
            12.0  2255.0  \\
            13.0  2269.0  \\
            14.0  1935.0  \\
            15.0  1193.0  \\
            16.0  953.0  \\
            17.0  1033.0  \\
            18.0  1859.0  \\
            19.0  1805.0  \\
            20.0  1253.0  \\
            21.0  2921.0  \\
            22.0  1085.0  \\
            23.0  875.0  \\
            24.0  3032.0  \\
            25.0  1419.0  \\
            26.0  1389.0  \\
            27.0  667.0  \\
            28.0  681.0  \\
            29.0  403.0  \\
            30.0  429.0  \\
};
\addlegendentry{\zerofpr{}}

\addplot[
    color=darkgreen,
    mark=triangle*,
]
table[row sep=\\]{
\\
            1.0  4015.0  \\
            2.0  3581.0  \\
            3.0  2865.0  \\
            4.0  2860.0  \\
            5.0  2636.0  \\
            6.0  2304.0  \\
            7.0  2238.0  \\
            8.0  2209.0  \\
            9.0  2040.0  \\
            10.0  1931.0  \\
            11.0  2389.0  \\
            12.0  1640.0  \\
            13.0  1612.0  \\
            14.0  1408.0  \\
            15.0  1044.0  \\
            16.0  864.0  \\
            17.0  753.0  \\
            18.0  928.0  \\
            19.0  1038.0  \\
            20.0  860.0  \\
            21.0  858.0  \\
            22.0  748.0  \\
            23.0  651.0  \\
            24.0  561.0  \\
            25.0  430.0  \\
            26.0  320.0  \\
            27.0  299.0  \\
            28.0  253.0  \\
            29.0  316.0  \\
            30.0  257.0  \\
};
\addlegendentry{\zerofprpp{}}

\nextgroupplot[
    title={Warm start},
    yticklabels={},
]

\addplot[
    color=darkblue,
    mark=*,
]
table[row sep=\\]{
\\
            1.0  11496.0  \\
            2.0  9529.0  \\
            3.0  4871.0  \\
            4.0  4911.0  \\
            5.0  4706.0  \\
            6.0  4023.0  \\
            7.0  3819.0  \\
            8.0  4751.0  \\
            9.0  3814.0  \\
            10.0  3681.0  \\
            11.0  4530.0  \\
            12.0  2775.0  \\
            13.0  2693.0  \\
            14.0  2547.0  \\
            15.0  2651.0  \\
            16.0  1667.0  \\
            17.0  1625.0  \\
            18.0  2021.0  \\
            19.0  1808.0  \\
            20.0  1617.0  \\
            21.0  1155.0  \\
            22.0  725.0  \\
            23.0  782.0  \\
            24.0  553.0  \\
            25.0  81.0  \\
            26.0  93.0  \\
            27.0  84.0  \\
            28.0  84.0  \\
            29.0  84.0  \\
            30.0  60.0  \\
};

\addplot[
    color=darkorange,
    mark=square*,
]
table[row sep=\\]{
\\
            1.0  8380.0  \\
            2.0  5996.0  \\
            3.0  3231.0  \\
            4.0  4155.0  \\
            5.0  4649.0  \\
            6.0  2751.0  \\
            7.0  2279.0  \\
            8.0  2615.0  \\
            9.0  2317.0  \\
            10.0  2303.0  \\
            11.0  2711.0  \\
            12.0  1939.0  \\
            13.0  1695.0  \\
            14.0  1599.0  \\
            15.0  1003.0  \\
            16.0  1071.0  \\
            17.0  833.0  \\
            18.0  859.0  \\
            19.0  785.0  \\
            20.0  815.0  \\
            21.0  751.0  \\
            22.0  447.0  \\
            23.0  447.0  \\
            24.0  398.0  \\
            25.0  55.0  \\
            26.0  51.0  \\
            27.0  59.0  \\
            28.0  57.0  \\
            29.0  53.0  \\
            30.0  57.0  \\
};

\addplot[
    color=darkgreen,
    mark=triangle*,
]
table[row sep=\\]{
\\
            1.0  4015.0  \\
            2.0  3462.0  \\
            3.0  2410.0  \\
            4.0  2504.0  \\
            5.0  2351.0  \\
            6.0  1902.0  \\
            7.0  1914.0  \\
            8.0  1985.0  \\
            9.0  1964.0  \\
            10.0  1723.0  \\
            11.0  2016.0  \\
            12.0  1487.0  \\
            13.0  1382.0  \\
            14.0  1298.0  \\
            15.0  893.0  \\
            16.0  836.0  \\
            17.0  764.0  \\
            18.0  785.0  \\
            19.0  819.0  \\
            20.0  777.0  \\
            21.0  573.0  \\
            22.0  379.0  \\
            23.0  455.0  \\
            24.0  345.0  \\
            25.0  60.0  \\
            26.0  56.0  \\
            27.0  60.0  \\
            28.0  58.0  \\
            29.0  57.0  \\
            30.0  61.0  \\
};

\end{groupplot}

\node[anchor=south] at ($(group c1r1.north)!0.5!(group c2r1.north)+(0,0.8cm)$)
{\pgfplotslegendfromname{grouplegendCS}};

\end{tikzpicture}}
		\caption{Performance (function + gradient calls) per MPC time step for the bicycle with horizon $N = 32$. Cold (left) and warm (right) starts.}
		\label{fig:bicycle}
	\end{figure}
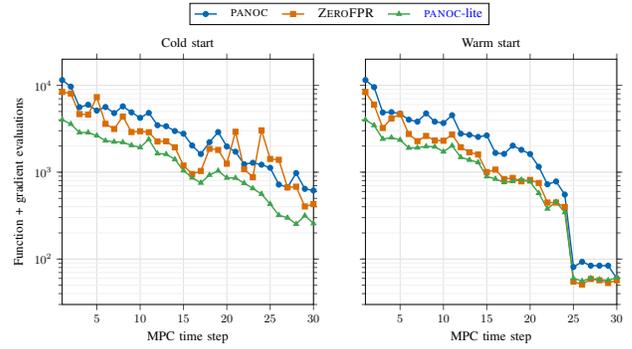

\subsection{Sparse logistic regression}\label{sec:logistic-regression}%
	Next, consider sparse logistic regression problems \eqref{eq:problem-statement} with
	\begin{equation*}
		f(x) = \sum_{i = 1}^m \ln \left( 1 + \exp^{-b_i \langle a_i, x \rangle} \right), \quad g(x) = \lambda \norm{ x }_1.
	\end{equation*}
	Here $\lambda > 0$ is a sparsity inducing parameter, the vector \(a_i \in \R^n\) contains the features of the $i$-th data point, and \(b_i \in \{-1, 1\}\) is the corresponding class label.
	Let us denote by \(A \in \R^{m \times n}\) the matrix with $a_i^\top$ as $i$-th row, and by \(b \in \R^m\) the vector with $i$-th component $b_i$.
	Note that the optimal solution for \(\lambda \geq \lambda_{\max} \eqdef \frac{1}{2} \norm{ A^\top b }_{\infty}\) is \(x^\star = 0\).
	\begin{figure}[h]
	    \centering
	    \resizebox{\linewidth}{!}{\input{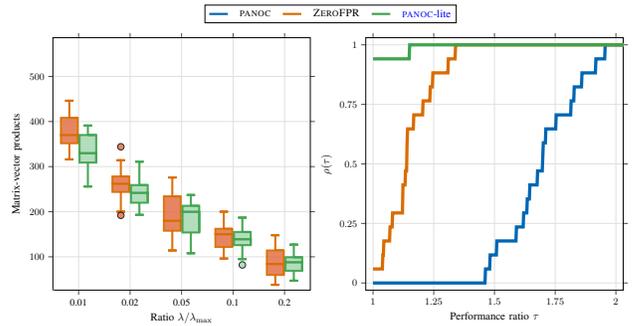}}
	    \caption{Performance comparison on a subset of the \textsc{Libsvm} problems. Left: Box plot of the number of matrix-vector products ($A$ and $A^\top$) for different choices of \(\lambda\). Boxes extend from the first quartile ($Q_1$) to the third quartile ($Q_3$), and the median value is indicated by the middle line. The whiskers extend to 1.5 times the interquartile range, i.e., $1.5 (Q_3 - Q_1)$. Right: Performance profile of the number of matrix-vector products ($A$ and $A^\top$) needed for \(\lambda = 0.01 \cdot \lambda_{\max}\).}%
	    \label{fig:logistic-regression}%
	\end{figure}
	We compare \cref{alg:zerofprpp} against \zerofpr{} and \panoc{} on a subset of \textsc{Libsvm} \cite{chang2011libsvm} (problems \texttt{a1a,\dots,a9a} and \texttt{w1a,\dots,w8a} to be precise), and this for \(\nicefrac{\lambda}{\lambda_{\max}} \in \{0.01, 0.02, 0.05, 0.1, 0.2\}\).
	All algorithms use L-BFGS directions with buffer size $5$.
	The Lipschitz constant is estimated as before.
	Performance is measured in terms of the total number of matrix-vector products, and all methods are terminated whenever \(\norm{ \Rnat{\gamma} }_\infty \leq 10^{-6}\).

	Figure \ref{fig:logistic-regression} indicates that the sparse logistic regression problem becomes more difficult as \(\lambda\) becomes smaller.
	Interestingly, it also shows that \cref{alg:zerofprpp} performs relatively similar to \zerofpr{} if \(\lambda\) is large, but outperforms \zerofpr{} when \(\lambda\) is large, i.e., in the more difficult setting.
	It also visualizes a performance profile\footnote{The higher and more leftward, the better the solver. $\tau = 1$: fraction of problems for which solver is the best; $\tau \to \infty$: fraction of problems solved.} \cite{dolan2002benchmarking} for the difficult case \(\lambda = 0.01 \cdot \lambda_{\max}\).
	In around $95\%$ of the problems, \cref{alg:zerofprpp} performs best, and it always remains within a factor $1.25$ of the best-performing solver.

\subsection{CUTEst benchmark collection}
	\begin{figure}[h]
		\centering
		\resizebox{\linewidth}{!}{\input{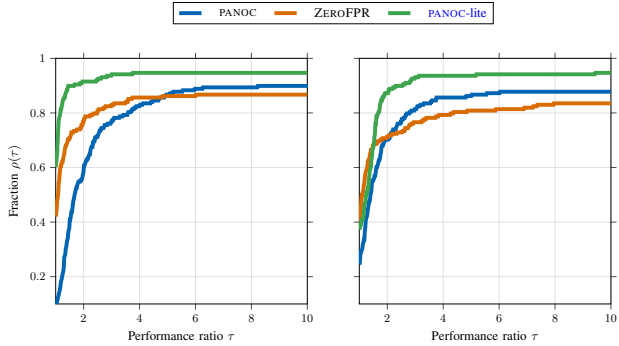}}
		\caption{Comparison of different ALM subproblem solvers within \alpaqa{}, applied to the \textsc{CUTEst} benchmarks. Performance profiles of the total number of iterations (left), and function + gradient evaluations (right) needed to solve the NLP.}
		\label{fig:cutest-performance}
	\end{figure}
	Finally, we consider problems from the \textsc{CUTEst} benchmarks.
	These are NLPs of the form \eqref{eq:nlp}, so we again use \alpaqa{} as described in \cref{sec:mpc}.
	Primal and dual tolerances are set to $10^{-3}$, and all methods use L-BFGS directions with buffer size $25$.

	\Cref{fig:cutest-performance} shows performance profiles in terms of the total number of iterations performed by the subproblem solver, and the total number of function and gradient evaluations, aggregated over the different ALM subproblems.
	We observe that \cref{alg:zerofprpp} outperforms \zerofpr{} and \panoc{} in terms of both performance metrics.

	\section{Conclusion}
		This work proposed a novel merit function for composite minimization problems, which, unlike existing alternatives, can be computed without gradient evaluations of the smooth objective term.
We then leveraged this merit function to design a linesearch method that improves upon state-of-the-art alternatives in the following ways:
(i) additional backtracking steps require no extra gradient evaluations,
and (ii) the method supports larger stepsizes, compatible to those of the standard proximal-gradient method.
After establishing global subsequential and local superlinear convergence of the method, we demonstrated the proposed method's effectiveness on nonlinear model predictive control problems with collision avoidance constraints, as well as on the \textsc{LIBSVM} and \textsc{Cute}st benchmarks.


	\bibliographystyle{IEEEtranS}
	\bibliography{TeX/references_abbr.bib}

	\appendix
		\setcounter{theorem}{0}%
		\renewcommand{\thetheorem}{A.\arabic{theorem}}%
		\renewcommand{\theHtheorem}{A.\arabic{theorem}}
		\begin{lemma}\label{thm:QUB}%
	Suppose that \cref{assumption:basic} holds and let \(x^\star\) be a fixed point of \(\prox_{\gamma g}(\id - \gamma \nabla f)\) for some \(\gamma>0\).
	Then, for any \(x\in\R^n\) it holds that
	\[
		\psi_\gamma(x)
	\leq
		\varphi(x^\star)
		+
		\tfrac{1+\gamma L_f}{2\gamma}\Vert x-\bar x^\star\Vert^2
		-
		\tfrac{1}{2\gamma}\Vert x-\hat x+\gamma\nabla f(x^\star)\Vert^2,
	\]
	where \(\psi_\gamma:=\varphi\circ\prox_{\gamma g}\) and \(\bar x^\star:=x^\star - \gamma \nabla f(x^\star)\).
\end{lemma}
\begin{proof}
	From \(x^\star=\prox_{\gamma g}(x^\star - \gamma \nabla f(x^\star))\), it follows that
	\begin{equation*}
		\begin{aligned}
			\nabla g^\gamma(\bar x^\star)
		&= \tfrac{1}{\gamma}\left[
			x^\star - \gamma \nabla f(x^\star)
			-
			\prox_{\gamma g}(x^\star - \gamma \nabla f(x^\star))
		\right]\\
		&= -\nabla f(x^\star).
		\end{aligned}
	\end{equation*}
	Then, by \(L_f\)- and \(\gamma^{-1}\)-smoothness of \(f\) and \(g^\gamma\), respectively,
  \begin{equation*}
    \begin{aligned}
      \psi_\gamma(&x) = f(\hat x) + g^\gamma(x) - \tfrac{1}{2\gamma}\Vert x-\hat x \Vert^2\\
      \leq{} &f(x^\star) + \langle \nabla f(x^\star), \hat x-x^\star \rangle + \tfrac{L_f}{2}\Vert \hat x-x^\star\Vert^2\\
      &+ g^\gamma(x^\star - \gamma \nabla f(x^\star)) - \langle \nabla f(x^\star),x-x^\star+\gamma\nabla f(x^\star) \rangle\\
      &+ \tfrac{1}{2\gamma}\Vert x^\star - \gamma \nabla f(x^\star)-x \Vert^2 - \tfrac{1}{2\gamma}\Vert x-\hat x\Vert^2\\
      ={} &\varphi(x^\star) - \tfrac{\gamma}{2}\Vert \nabla f(x^\star) \Vert^2 + \langle \nabla f(x^\star), \hat x-x \rangle\\ &+ \tfrac{L_f}{2}\Vert \hat x-x^\star \Vert^2 - \tfrac{1}{2\gamma}\Vert x-\hat x \Vert^2 + \tfrac{1}{2\gamma}\Vert \bar x^\star-x\Vert^2,
    \end{aligned}
  \end{equation*}
	where the last step used
	\(
		g^\gamma(x^\star - \gamma \nabla f(x^\star))
	=
    g(x^\star)
		+
		\tfrac{\gamma}{2}\Vert \nabla f(x^\star) \Vert^2.
	\)
	The claim follows by completing the squares and from nonexpansiveness of \(\prox_{\gamma g}\) (recall that \(x^\star=\prox_{\gamma g}(\bar x^\star)\)).
\end{proof}


\end{document}